\documentclass[11pt,a4]{amsart}
\usepackage[centertags]{amsmath}
\usepackage{amsfonts}
\usepackage{amssymb}
\usepackage{amsthm}
\usepackage{newlfont}
\newcommand{\K}{\Bbb K}
\newcommand{\C}{\Bbb C}
\newcommand{\Z}{\Bbb Z}
\newcommand{\N}{\Bbb N}
\newcommand{\Q}{\Bbb Q}

\def\NL{\hfill\break}
\def\ni{\noindent}
\newcommand{\set}[1]{\left\{#1\right\}}
\newtheorem{theo}{Theorem}[section]
\newtheorem{cor}[theo]{Corollary}
\newtheorem{lem}[theo]{Lemma}
\newtheorem{prop}[theo]{Proposition}
\newtheorem{ex}[theo]{Example}

\begin{document}
\title{On quasi-Baer rings of Ore extensions}
\date{}
\author{L'moufadal Ben Yakoub and Mohamed Louzari}
\address{L'moufadal Ben yakoub\\
Department of mathematics\\ University Abdelmalek Essaadi\\ B.P.
2121 Tetouan, Morocco} \email{benyakoub@hotmail.com}
\address{Mohamed Louzari\\
Department of mathematics\\ University Abdelmalek Essaadi\\ B.P.
2121 Tetouan, Morocco} \email{louzari\_mohamed@hotmail.com}

\normalsize
\begin{abstract}Let $R$ be a ring and $S=R[x;\sigma,\delta]$ its Ore
extension. We prove under some conditions that $R$ is a quasi-Baer
ring if and only if the Ore extension $R[x;\sigma,\delta]$ is a
quasi-Baer ring. Examples are provided to illustrate and delimit our
results.
\end{abstract}

\maketitle

\ni\footnote[0]{{\it Key words and phrases.} Ore extensions, Quasi-Baer
rings, Skew Armendariz rings.\NL This work was partially supported
by the integrated action Moroccan-Spanish 62/P/03.\NL This paper appeared in 
East-West J. of Mathematics: Vol. 8, No 2 (2006) pp. 119-127}

\section{Introduction} Throughout this paper, $R$ denotes
an associative ring with unity. For a subset $X$ of $R$,
$r_R(X)=\{a\in R|Xa=0\}$ and $\ell_R(X)=\{a\in R|aX=0\}$ will stand
for the right and the left annihilator of $X$ in $R$ respectively.
By \cite{Kaplansky}, a right annihilator of $X$ is always a right
ideal, and if $X$ is a right ideal then $r_R(X)$ is a two-sided
ideal. An Ore extension of a ring $R$ is denoted by
$R[x;\sigma,\delta]$, where $\sigma$ is an endomorphism of $R$ and
$\delta$ is a $\sigma$-derivation, i.e., $\delta\colon R\rightarrow
R$ is an additive map such that
$\delta(ab)=\sigma(a)\delta(b)+\delta(a)b$ for all $a,b\in R$.
Recall that elements of $R[x;\sigma,\delta]$ are polynomials in $x$
with coefficients written on the left. Multiplication in
$R[x;\sigma,\delta]$ is given by the multiplication in $R$ and the
condition $xa=\sigma(a)x+\delta(a)$, for all $a\in R$. We say that a
subset $X$ of $R$ is $(\sigma,\delta)$-{\it stable} if
$\sigma(X)\subseteq X$ and $\delta(X)\subseteq X$. A ring $R$ is
({\it quasi})-{\it Baer} if the right annihilator of every nonempty
subset (every right ideal) of $R$ is generated by an idempotent.
From \cite{birk1}, an idempotent $e\in R$ is left (resp. right) {\it
semicentral} in $R$ if $exe=xe$ (resp. $exe=ex$), for all $x\in R$.
Equivalently, $e^2=e\in R$ is left (resp. right) semicentral if $eR$
(resp. $Re$) is an ideal of $R$. Since the right annihilator of a
right ideal is an ideal, we see that the right annihilator of a
right ideal is generated by a left semicentral in a quasi-Baer ring.
We use $\mathcal{S}_\ell(R)$ and $\mathcal{S}_r(R)$ for the sets of
all left and right semicentral idempotents, respectively. Also note
$\mathcal{S}_\ell(R)\cap \mathcal{S}_r(R)=\mathcal{B}(R)$, where
$\mathcal{B}(R)$ is the set of all central idempotents of $R$. If
$R$ is a semiprime ring then
$\mathcal{S}_\ell(R)=\mathcal{S}_r(R)=\mathcal{B}(R)$. Recall that
$R$ is a {\it reduced} ring if it has no nonzero nilpotent elements.
A ring $R$ is {\it abelian} if every idempotent of $R$ is central.
We can easily observe that every reduced ring is abelian. \par
According to \cite{krempa}, an endomorphism $\sigma$ of a ring $R$
is said to be {\it rigid} if $a\sigma(a)=0$ implies $a=0$ for all
$a\in R$. We call a ring $R$ $\sigma$-{\it rigid} if there exists a
rigid endomorphism $\sigma$ of $R$. Following Hashemi and Moussavi
\cite{Hashemi/polext}, a ring $R$ is $\sigma$-{\it compatible} if
for each $a,b\in R$, $a\sigma(b)=0\Leftrightarrow ab=0$. Moreover,
$R$ is said to be $\delta$-{\it compatible} if for each $a,b\in R$,
$ab=0\Rightarrow a\delta(b)=0$. If $R$ is both $\sigma$-compatible
and $\delta$-compatible, we say that $R$ is $(\sigma,\delta)$-{\it
compatible}. A ring $R$ is $\sigma$-rigid if and only if $R$ is
$(\sigma,\delta)$-compatible and reduced
\cite[Lemma~2.2]{Hashemi/polext}. Also, if $R$ is $\sigma$-rigid
then $R[x;\sigma,\delta]$ is reduced \cite[Theorem~3.3]{krempa}.
From \cite{hong2}, a ring $R$ is said to be a $\sigma$-{\it skew
Armendariz } ring if for $p=\sum_{i=0}^{n}a_ix^i$ and
$q=\sum_{j=0}^{m}b_jx^j$ in $R[x;\sigma]$, $pq=0$ implies
$a_i\sigma^i(b_j)=0$ for all $0\leq i\leq n$ and $0\leq j\leq m$.
From \cite{Hashemi/skewarm}, a ring $R$ is called an
$(\sigma,\delta)$-{\it skew Armendariz } ring if for
$p=\sum_{i=0}^{n}a_ix^i$ and $q=\sum_{j=0}^{m}b_jx^j$ in
$R[x;\sigma,\delta]$, $pq=0$ implies $a_ix^ib_jx^j=0$ for each
$i,j$. Note that $(\sigma,\delta)$-skew Armendariz rings are
generalization of $\sigma$-skew Armendariz rings, $\sigma$-rigid
rings and Armendariz rings, see \cite{hong2}, for more details. It
was proved in \cite[Corollary~12]{hong1}, that if $R$ is a
$\sigma$-rigid ring then $R[x;\sigma,\delta]$ is a quasi-Baer ring
if and only if $R$ is quasi-Baer.  Also in
\cite[Corollary~2.8]{Hashemi/polext}, it was shown that, if $R$ is
$(\sigma,\delta)$-compatible, then $R[x;\sigma,\delta]$ is a
quasi-Baer ring if and only if $R$ is quasi-Baer.\\

\par The aim of this paper is to show that if $R$ is an $(\sigma,\delta)$-skew
Armendariz ring with $\sigma$ an automorphism such that $Re$ is
$(\sigma,\delta)$-stable for all $e\in\mathcal{S}_\ell(R)$, then $R$
is a quasi-Baer ring if and only if $R[x;\sigma,\delta]$ is a
quasi-Baer ring. Many examples are provided to illustrate and
delimit results and to show that they are not consequences of
\cite[Corollary 2.8]{Hashemi/polext}. Moreover, we obtain a partial
generalization of \cite[Corollary 12]{hong1}.

\section{Preliminaries and Examples}
For any $0\leq i\leq j\;(i,j\in \N)$, $f_i^j\in End(R,+)$ will
denote the map which is the sum of all possible words in
$\sigma,\delta$ built with $i$ letters $\sigma$ and $j-i$ letters
$\delta$ (e.g., $f_n^n=\sigma^n$ and $f_0^n=\delta^n, n\in \N $).
The next lemma appears in \cite[Lemma 4.1]{Lam}.
\begin{lem}\label{lemma1}
For any $n\in \N$ and $r\in R$ we have
$x^nr=\sum\limits_{i=0}^nf_i^n(r)x^i$ in the ring
$R[x;\sigma,\delta]$.
\end{lem}

\begin{lem}\label{lemma2}\cite[Lemma 5]{Hashemi/skewarm}. Let $R$ be an $(\sigma,\delta)$-skew Armendariz ring. If
$e^2=e\in R[x;\sigma,\delta]$ where
$e=e_0+e_1x+e_2x^2+\cdots+e_nx^n$, then $e=e_0$.
\end{lem}

\begin{lem}\label{lemma3}Let $R$ be a ring, $\sigma$ an endomorphism and $\delta$
be a $\sigma$-derivation of $R$. Then $\sigma(Re)\subseteq Re$
implies $\delta(Re)\subseteq Re$ for all $e\in \mathcal{B}(R)$.
\end{lem}

\begin{proof}Let $e\in \mathcal{B}(R)$ and $r\in R$.
Then
$\delta(re)=\delta(ere)=\sigma(er)\delta(e)+\delta(er)e=\sigma(ere)\delta(e)+\delta(er)e=
se\delta(e)+\delta(er)e$, for some $s\in R$, but $e\in
\mathcal{B}(R)$, then $e\delta(e)=e\delta(e)e$, so
$\delta(re)=(se\delta(e)+\delta(er))e$. Therefore
$\delta(Re)\subseteq Re$.
\end{proof}

\begin{lem}\label{lemma4}Let $R$ be a ring, $\sigma$ an endomorphism of $R$
and $\delta$ be a $\sigma$-derivation of $R$. If $R$ is
$(\sigma,\delta)$-compatible. Then for $a,b\in R$, $ab=0 \Rightarrow
af_i^j(b)=0$ for all $j\geq i\geq 0$.
\end{lem}
\begin{proof}If $ab=0$, then $a\sigma^i(b)=a\delta^j(b)=0$ for all
$i\geq0$ and $j\geq 0$, because $R$ is $(\sigma,\delta)$-compatible.
Then $af_i^j(b)=0$ for all $i,j$.
\end{proof}

\begin{lem}\label{lemma5}Let $R$ be a ring, $\sigma$ an endomorphism of $R$
and $\delta$ be a $\sigma$-derivation of $R$. If $R$ is
$\sigma$-rigid then $R$ is $(\sigma,\delta)$-skew Armendariz.
\end{lem}

\begin{proof}If $R$ is $\sigma$-rigid then $R$ is
$(\sigma,\delta)$-compatible by \cite[Lemma~2.2]{Hashemi/polext}.
Let $f=\sum_{i=0}^na_ix^i,\;g=\sum_{j=0}^mb_jx^j\in
R[x;\sigma,\delta]$ such that $fg=0$, then $a_ib_j=0$ for all $i,j$,
by \cite[Proposition~ 6]{hong1}. So $a_if_{\ell}^j(b_j)=0$, for all
$0\leq\ell\leq i\leq n,\;0\leq j\leq m$, by Lemma \ref{lemma4}.
Hence $a_ix^ib_jx^j=\sum_{\ell=0}^ia_if_{\ell}^j(b_j)x^{\ell+j}=0$.
Therefore $R$ is $(\sigma,\delta)$-skew Armendariz.
\end{proof}

The next example illustrates that there exists a ring $R$ and an
automorphism $\sigma$ of $R$ such that $Re$ is $\sigma$-stable for
all $e\in \mathcal{S}_{\ell}(R)$, but $R$ is not $\sigma$-rigid.

\begin{ex}\label{ex1}\cite[Example~1]{hong2}. Consider the ring $$R=\set{\begin{pmatrix}
  a & t \\
  0 & a
\end{pmatrix}| a\in \Z\;,t\in \Q},$$ where $\Z$ and $\Q$ are the
set of all integers and all rational numbers, respectively. The ring
$R$ is commutative, let $\sigma\colon R\rightarrow R$ be an
automorphism defined by $\sigma\left(\begin{pmatrix}
  a & t \\
  0 & a
\end{pmatrix}\right)=\begin{pmatrix}
  a & t/2 \\
  0 & a
\end{pmatrix}$.\NL
$(1)$ $R$ is not $\sigma$-rigid.\NL $\begin{pmatrix}
  0 & t \\
  0 & 0
\end{pmatrix}\sigma\left(\begin{pmatrix}
  0 & t \\
  0 & 0
\end{pmatrix}\right)=0$, but $\begin{pmatrix}
  0 & t \\
  0 & 0
\end{pmatrix}\neq 0$, if $t\neq 0$.\NL
$(2)$ $\sigma(Re)\subseteq Re$ for all $e\in \mathcal{S}_{\ell}(R)$.
$R$ has only two idempotents:\NL
 $e_0=\begin{pmatrix}
  0 & 0 \\
  0 & 0
\end{pmatrix}$ end $e_1=\begin{pmatrix}
 1 & 0 \\
  0 & 1
\end{pmatrix}$, let $r=\begin{pmatrix}
  a & t \\
  0 & a
\end{pmatrix}\in R$, we have $\sigma(re_0)\in Re_0$ and $\sigma(re_1)\in
Re_1$.
\end{ex}

Also we have an example of an endomorphism $\sigma$ of a ring $R$
such that $Re$ is $\sigma$-stable for all $e\in
\mathcal{S}_{\ell}(R)$ and $R$ is not $\sigma$-compatible.

\begin{ex}\label{ex2}Let $\K$ be a field and $R=\K[t]$ a polynomial ring over $\K$
with the endomorphism $\sigma$ given by $\sigma(f(t))=f(0)$  for all
$f(t)\in R$.\NL$(1)$ $R$ is not $\sigma$-compatible (so not
$\sigma$-rigid). Take $f=a_0+a_1t+a_2t^2+\cdots+a_nt^n$ and
$g=b_1t+b_2t^2+\cdots+b_mt^m$, since $g(0)=0$ so, $f\sigma(g)=0$,
but $fg\neq 0$.\NL$(2)$ $R$ has only two idempotents $0$ and $1$ so
$Re$ is $\sigma$-stable for all $e\in\mathcal{S}_\ell(R)$.
\end{ex}

There is an example of a ring $R$ and an endomorphism $\sigma$ of
$R$ such that $R$ is $\sigma$-skew Armendariz and $R$ is not
$\sigma$-compatible.

\begin{ex}\label{ex3}Consider a ring of polynomials over $\Z_2$,
$R=\Z_2[x]$. Let $\sigma\colon R\rightarrow R$ be an endomorphism
defined by $\sigma(f(x))=f(0)$. Then: \NL $(i)$ $R$ is not
$\sigma$-compatible. Let $f=\overline{1}+x$, $g=x\in R$, we have
$fg=(\overline{1}+x)x\neq 0$, however
$f\sigma(g)=(\overline{1}+x)\sigma(x)=0$.\NL $(ii)$ $R$ is
$\sigma$-skew Armendariz \cite[Example~5]{hong2}.
\end{ex}

In the next example, $S=R/I$ is a ring and $\overline{\sigma}$ an
endomorphism of $S$ such that $S$ is $\overline{\sigma}$-compatible
and not $\overline{\sigma}$-skew Armendariz.

\begin{ex}\label{ex4} Let $\Z$ be the ring of integers and $\Z_2$ be
the ring of integers modulo 4. Consider the ring
$$R=\set{\begin{pmatrix}
  a & \overline{b} \\
  0 & a
\end{pmatrix}| a\in\Z\;,\overline{b}\in \Z_4}.$$  Let $\sigma\colon R\rightarrow R$ be an
endomorphism defined by $\sigma\left(\begin{pmatrix}
  a & \overline{b} \\
  0 & a
\end{pmatrix}\right)=\begin{pmatrix}
  a & -\overline{b} \\
  0 & a
\end{pmatrix}$.\NL Take the ideal $I=\set{\begin{pmatrix}
  a & \overline{0} \\
  0 &a
\end{pmatrix}| a\in 4\Z}$ of $R$. Consider the factor ring $$R/I\cong\set{\begin{pmatrix}
  \overline{a} & \overline{b} \\
  0 &\overline{a}
\end{pmatrix}| \overline{a},\overline{b}\in 4\Z}.$$ \NL$(1)$ $R/I$ is not
$\overline{\sigma}$-skew Armendariz. In fact, $\left(\begin{pmatrix}
 \overline{2} & \overline{0} \\
  0 & \overline{2}
\end{pmatrix}+\begin{pmatrix}
 \overline{2} & \overline{1} \\
  0 & \overline{2}
\end{pmatrix}x\right)^2=0\in(R/I)[x;\overline{\sigma}]$, but $\begin{pmatrix}
 \overline{2} & \overline{1} \\
  0 & \overline{2}
\end{pmatrix}\overline{\sigma}\begin{pmatrix}
 \overline{2} & \overline{0} \\
  0 & \overline{2}
\end{pmatrix}\neq 0$.\NL$(2)$ $R/I$ is
$\overline{\sigma}$-compatible. Let $A=\begin{pmatrix}
 \overline{a} & \overline{b} \\
  0 & \overline{a}
\end{pmatrix}\;,B=\begin{pmatrix}
 \overline{a'} & \overline{b'} \\
  0 & \overline{a'}
\end{pmatrix}\in R/I$. If $AB=0$ then $\overline{aa'}=0$ and
$\overline{ab'}=\overline{ba'}=0$, so that
$A\overline{\sigma}(B)=0$. The same for the converse. Therefore
$R/I$ is $\overline{\sigma}$-compatible.
\end{ex}

\section{Ore extensions over quasi-Baer rings}
It was proved in \cite[Theorem~1.2]{birk1}, that if $R$ is a
quasi-Baer ring and $\sigma$ an automorphism of $R$ then
$R[x;\sigma]$ is a quasi-Baer ring. The following example shows that
`` $\sigma$ is an automorphism " is not a superfluous condition in
Proposition \ref{mainprop}.

\begin{ex}\cite[Example~2.8]{birk4}. There is an example of a quasi-Baer
ring $R$ and an endomorphism $\sigma$ of $R$ such that $R[x;\sigma]$
is not a quasi-Baer ring. In fact, let $R=\K[t]$ be the polynomial
ring over a field $\K$ and $\sigma$ be the endomorphism given by
$\sigma(f(t))=f(0)$. Then the ring $R[x;\sigma]$ is not a quasi-Baer
ring.
\end{ex}

\begin{prop}\label{mainprop}Let $R$ be a ring, $\sigma$ an automorphism and
$\delta$ be a $\sigma$-derivation of $R$. Suppose that $Re$ is
$(\sigma,\delta)$-stable for all $e\in\mathcal{S}_\ell(R)$. If $R$
is quasi-Baer then the Ore extension $R[x;\sigma,\delta]$ is
quasi-Baer.
\end{prop}
\begin{proof}Let $S=R[x;\sigma,\delta]$ and $I$ be an ideal of $S$. We
claim that $r_S(I)=eS$, for some idempotent $e\in R$. We can suppose
that $I\neq 0$, we set \NL $I_0=\{0\}\cup\{a\in R\;|\;\exists\;
a_0,a_1,\cdots,a_{n-1}\in R$ such that $ax^n+\sum\limits_{i=0}^{n-1}
a_ix^i\in I,n\in \N\}$. It is clear that $I_0$ is a nonzero left
ideal of $R$. Given $a\in I_0$ and $r\in R$, there is an element in
$I$ of the form $ax^n+\sum\limits_{i=0}^{n-1} a_ix^i$. Multiplying
on the right by $\sigma^{-n}(r)$ gives an element of the form
$arx^n+\sum\limits_{i=0}^{n-1} b_ix^i$, for some elements
$b_0,b_1,\cdots,b_{n-1}\in R$, and so $ar\in I_0$, thus $I_0$ is a
two-sided ideal. So there exists an idempotent $e\in R$ such that
$r_R(I_0)=eR$. We have $eS\subseteq r_S(I)$. To see this, let $0\neq
f(x)=\sum\limits_{k=0}^{n} a_kx^k\in I$, then
$f(x)e=\sum\limits_{k=0}^{n}(\sum\limits_{i=k}^{n}a_kf_k^i(e))x^k$,
where $f_k^i$ are sums of all possible words in $\sigma,\delta$
built with $k$ letters $\sigma$ and $i-k$ letters $\delta$. $Re$ is
$f_k^i$-stable $(0\leq k\leq i)$, so there exists $u_k^i\in R$ such
that $f_k^i(e)=u_k^ie\;(0\leq k\leq i)$. Therefore
$f(x)e=\sum\limits_{k=0}^{n}(\sum\limits_{i=k}^{n}a_ku_k^i)ex^k$, if
we set $\alpha_k=\sum\limits_{i=k}^{n}a_ku_k^ie$, then
$f(x)e=\sum\limits_{k=0}^{n}\alpha_{k}x^{k}$. If $\alpha_n\neq 0$,
then $\alpha_n\in I_0$ and so, $\alpha_ne=\alpha_n=0$ ( because
$r_R(I_0)=eR$ ). Contradiction, hence $\alpha_n=0$. Now suppose that
$\alpha_j=0$ for $j=n,n-1,\cdots,k+1$ with $k\in \N$. But
$f(x)e=\alpha_kx^k+\sum\limits_{\ell=0}^{k-1}\alpha_{\ell}x^{\ell}$,
with the same manner as above we have $\alpha_k=0$. So we can get
$\alpha_n=\alpha_{n-1}=\cdots=\alpha_0=0$. Consequently $eS\subseteq
r_S(I)$.\par Conversely, we can claim that $r_S(I)\subseteq eS$. Let
$0\neq f(x)=\sum\limits_{k=0}^{n} a_kx^k\in I$ and
$\lambda(x)=\sum\limits_{j=0}^mb_jx^j\in S$, such that
$f(x)\lambda(x)=0$, we shall show that
$\lambda(x)=\sigma^{-n}(e)\lambda(x)$. If we set
$\xi(x)=\lambda(x)-\sigma^{-n}(e)\lambda(x)=\sum\limits_{j=0}^m(b_j-\sigma^{-n}(e)b_j)x^j$,
we have
$f(x)\xi(x)=(\sum\limits_{i=0}^na_ix^i)(\sum\limits_{j=0}^m(b_j-\sigma^{-n}(e)b_j)x^j)=
a_n\sigma^n(b_m-\sigma^{-n}(e)b_m)x^{n+m}+Q=0$, where $Q$ is a
polynomial with $deg(Q)<n+m$. Thus
$a_n\sigma^n(b_m-\sigma^{-n}(e)b_m)=0$, since $a_n\neq 0$, then
$a_n\in I_0$. Hence $\sigma^n(b_m-\sigma^{-n}(e)b_m)\in
r_R(I_0)=eR$. So
$\sigma^n(b_m-\sigma^{-n}(e)b_m)=e\sigma^n(b_m-\sigma^{-n}(e)b_m)$,
then
$b_m-\sigma^{-n}(e)b_m=\sigma^{-n}(e)(b_m-\sigma^{-n}(e)b_m)=0)$
(because $\sigma^{-n}(e)$ is idempotent), hence
$b_m-\sigma^{-n}(e)b_m=0$. Now, suppose that
$b_j-\sigma^{-n}(e)b_j=0$ for $j=m,m-1,\cdots,k+1$ with $k\in \N$
and showing that $b_k-\sigma^{-n}(e)b_k=0$. Effectively,
$f(x)\xi(x)=a_n\sigma^{n}(b_k-\sigma^{-n}(e)b_k)x^{n+k}+Q'=0$, where
$Q'$ is a polynomial with $deg(Q')<n+k$, then
$a_n\sigma^{n}(b_k-\sigma^{-n}(e)b_k)=0$, with the same manner as
below, we obtain $b_k-\sigma^{-n}(e)b_k=0$. Therefore
$b_j-\sigma^{-n}(e)b_j=0$ for all $0\leq j\leq m$, then $\xi(x)=0$.
But $\lambda(x)=\sigma^{n}(e)\lambda(x)$ or $\sigma^{n}(e)=ue$ for
some $u\in R$, but $e$ is left semicentral then
$\lambda(x)=eue\lambda(x)$ . Hence $r_S(I)\subseteq eS$. So
$R[x;\sigma,\delta]$ is a quasi-Baer ring.
\end{proof}


In Example \ref{ex2}, $Re$ is $(\sigma,\delta)$-stable for all
$e\in\mathcal{S}_\ell(R)$ but $R$ is not
$(\sigma,\delta)$-compatible. Thus, Proposition \ref{mainprop} is
not a consequence of \cite[Corollary~2.8]{Hashemi/polext}.\\

There is a quasi-Baer ring $R$, $\sigma$ an automorphism of $R$ and
$\delta$ a $\sigma$-derivation of $R$ such that $Re$ is
$(\sigma,\delta)$-stable for all $e\in\mathcal{S}_\ell(R)$.

\begin{ex}\label{ex5}Consider the ring $R=\begin{pmatrix}
  \Z & \Z \\
  0 & \Z
\end{pmatrix},$ where $\Z$ is the set of all integers numbers. By
\cite[Example 1.3(ii)]{birk2}, $R$ is a quasi-Baer ring. Define
$\sigma\colon R \rightarrow R$ and $\delta\colon R \rightarrow R$ by
$$\sigma\left(\begin{pmatrix}
  a & b \\
  0 & c
\end{pmatrix}\right)=\begin{pmatrix}
  a& -b \\
  0 & c
\end{pmatrix},\;\;\delta\left(\begin{pmatrix}
  a & b \\
  0 & c
\end{pmatrix}\right)=\begin{pmatrix}
  0 & 2b \\
  0 & 0
\end{pmatrix} \mathrm {\;for\; all\;} a,b,c\in \Z.$$ Clearly,
$\sigma$ is an automorphism of $R$ and $\delta$ is a
$\sigma$-derivation. The nonzero idempotents of $R$ are of the form
$$e_0=\begin{pmatrix}
  1 & 0 \\
  0 & 1
\end{pmatrix},\;e_1=\begin{pmatrix}
  1 & t \\
  0 & 0
\end{pmatrix}\;\mathrm{and}\;e_2=\begin{pmatrix}
  0 & t \\
  0 & 1
\end{pmatrix},$$ where $t\in\Z$. $e_2$ is right semicentral not
left semicentral and $e_1$ is left semicentral not right
semicentral, so the only left semicentral nonzero idempotents of $R$
are $e_0$ and $e_1$. $Re_0$ is $(\sigma,\delta)$-stable. Let
$r=\begin{pmatrix}
  x & y \\
  0 & z
\end{pmatrix}\in R$, since $\sigma(re_1)=\begin{pmatrix}
  x & -xt \\
  0 & 0
\end{pmatrix}\in \begin{pmatrix}
  \Z & \Z \\
  0 & \Z
\end{pmatrix}\begin{pmatrix}
  1 & t \\
  0 & 0
\end{pmatrix}$, then $Re_1$ is $\sigma$-stable, also $Re_1$ is $\delta$-stable.
Therefore $Re$ is $(\sigma,\delta)$-stable for all
$e\in\mathcal{S}_\ell(R)$.
\end{ex}

\begin{ex}\label{ex6} Consider the ring $S=\begin{pmatrix}
  D & D\oplus D \\
  0 & D
\end{pmatrix}$, where $D$ is a simple domain which is not a
division ring. By \cite[Example 4.11]{birk3}, $R$ is a quasi-Baer
ring and has nonzero idempotents of the form
$$\begin{pmatrix}
  1 & 0 \\
  0 & 1
\end{pmatrix},\;\begin{pmatrix}
  1 & (b,d) \\
  0 & 0
\end{pmatrix}\;\mathrm{and}\;\begin{pmatrix}
  0 & (b,d) \\
  0 & 1
\end{pmatrix},$$ where $b,d\in D$, with $\sigma$ and $\delta$ as
in Example \ref{ex5}, $Re$ is $(\sigma,\delta)$-stable for all
$e\in\mathcal{S}_\ell(R)$.
\end{ex}

\begin{cor} Let $R$ be an abelian or a semiprime ring, $\sigma$ an automorphism
and $\delta$ be a $\sigma$-derivation of $R$, such that
$\sigma(Re)\subseteq Re$ for all $e\in \mathcal{B}(R)$. If $R$ is
quasi-Baer then $R[x;\sigma,\delta]$ is quasi-Baer.
\end{cor}

\begin{proof}By Lemma \ref{lemma3} and
Proposition \ref{mainprop}.
\end{proof}

In the remainder of this section we focus on the converse of
Proposition \ref{mainprop}. We begin with the next example which
shows that there exists a ring $R$ and a derivation $\delta$ of $R$
such that $R[x;\delta]$ is quasi-Baer but $R$ is not quasi-Baer.

\begin{ex}\label{ex7}\cite[Example 1.6]{birk1}. There is a ring $R$ and a derivation
$\delta$ of $R$ such that $R[x;\delta]$ is a Baer ring. But $R$ is
not quasi-Baer. Let $R=\Z_2[t]/(t^2)$ with the derivation $\delta$
such that $\delta(\overline{t})=1$ where $\overline{t}=t+(t^2)$ in
$R$ and $\Z_2[t]$ is the polynomial ring over the field $\Z_2$ of
two elements. Consider the Ore extension $R[x;\delta]$. If we set
$e_{11}=\overline{t}x,e_{12}=\overline{t},e_{21}=\overline{t}x^2+x$
and $e_{22}=1+\overline{t}x$ in $R[x;\delta]$, then they form a
system of matrix units in $R[x;\delta]$. Now the centralizer of
these matrix units in $R[x;\delta]$ is $\Z_2[x^2]$. Therefore
$R[x;\delta]\cong M_2(\Z_2[x^2])\cong M_2(\Z_2)[y]$, where
$M_2(\Z_2)[y]$ is the polynomial ring over $M_2(\Z_2)$. So the ring
$R[x;\delta]$ is a Baer ring, but $R$ is not quasi-Baer.
\end{ex}

\begin{prop}\label{prop2}Let $R$ be an $(\sigma,\delta)$-skew Armendariz ring. If $R[x;\sigma,\delta]$ is
quasi-Baer then $R$ is quasi-Baer.
\end{prop}

\begin{proof}Let $I$ be an ideal of $R$ and $S=R[x;\sigma,\delta]$, then
since $S$ is quasi-Baer, there exists an idempotent $e\in S$ such
that $r_S(IS)=eS$ with $e=e_0+e_1x+\cdots+e_nx^n\;(n\in\N)$. By
Lemma \ref{lemma2}, we have $e_0\in r_R(I)$. Thus $e_0R\subseteq
r_R(I)$.\par Conversely, let $a\in r_R(I)$ then $a\in r_S(IS)\cap
R=e_0S\cap R$, so $a=e_0f$ for some $f=f_0+f_1x+\cdots+f_mx^m\in S$.
Then $a=e_0f_0$ and so $a\in e_0R$. Therefore $r_R(I)\subseteq
e_0R$. Consequently, $R$ is a quasi-Baer ring.
\end{proof}

By Example \ref{ex3}, there is a ring $R$ and $\sigma$ an
endomorphism of $R$ such that $R$ is $\sigma$-skew Armendariz and
$R$ is not $\sigma$-compatible. So that, Proposition \ref{prop2} is
not a consequence of \cite[Corollary 2.8]{Hashemi/polext}. By the
next result, we see that Proposition \ref{prop2} is a partial
generalization of \cite[Corollary 12]{hong1}.

\begin{cor}Let $R$ be an $\sigma$-rigid ring. If $R[x;\sigma,\delta]$
is quasi-Baer then $R$ is quasi-Baer.
\end{cor}
\begin{proof}It follows from Lemma \ref{lemma5} and
Proposition \ref{prop2}.
\end{proof}

One might expect the converse of Proposition \ref{mainprop} to hold
when $R$ is a $(\sigma,\delta)$-skew Armendariz ring. However
\cite[Example 5]{hong2} and \cite[Example 2.8]{birk4}, shows that
this converse does not hold in general.
\begin{ex}\label{v}We consider a commutative polynomial ring over $\Z_2$.
$R=\Z_2[x]$, let $\sigma\colon R\rightarrow R$ be an endomorphism
defined by $\sigma(f(x))=f(0)$. By \cite[Example 2.8]{birk4},
$R[x;\sigma]$ is not quasi-Baer and $R$ is quasi-Baer. But, by
\cite[Example 5]{hong2}, $R$ is $\sigma$-skew Armendariz. Note that
$R$ has only two idempotents $0$ and $1$, so $\sigma(Re)\subseteq
Re$ for all $e\in \mathcal{S}_\ell(R)$. Thus `` $\sigma$ is an
automorphism " is not a superfluous condition in the next theorem.
\end{ex}

\begin{theo}\label{theo} Let $R$ be a $(\sigma,\delta)$-skew Armedariz ring
with $\sigma$ an automorphism such that $Re$ is
$(\sigma,\delta)$-stable for all $e\in\mathcal{S}_\ell(R)$. Then $R$
is a quasi-Baer ring if and only if $R[x;\sigma,\delta]$ is a
quasi-Baer ring.
\end{theo}

\begin{proof} It follows immediately from Proposition \ref{mainprop} and
Proposition \ref{prop2}.
\end{proof}

\begin{ex} Let $R=\C$ where $\C$ is the field of complex numbers.
Then $R$ is a Baer (so quasi-Baer) reduced ring. Define
$\sigma\colon R\rightarrow R$ and $\delta\colon R\rightarrow R$ by
$\sigma(z)=\overline{z}$ and $\delta(z)=z-\overline{z}$, where
$\overline{z}$ is the conjugate of $z$. $\sigma$ is an automorphism
of $R$ and $\delta$ is a $\sigma$-derivation. $R$ has only two
idempotents $0$ and $1$, so we have the stability indicated in
Theorem \ref{theo}.\par We claim that $R$ is a
$(\sigma,\delta)$-skew Armendariz ring. Consider
$R[x;\sigma,\delta]$. Let $p=a_0+a_1x+\cdots+a_nx^n$ and
$q=b_0+b_1x+\cdots+b_mx^m\in R[x;\sigma,\delta]$. Assume that
$pq=0$. Since $R$ is $\sigma$-rigid, we have $a_ib_j=0$ for all
$0\leq i \leq n$ and $0\leq j \leq m$, by \cite[Proposition
6]{hong1}. thus $a_ix^ib_jx^j=0$ for all $0\leq i \leq n$ and $0\leq
j \leq m$, because $R[x;\sigma,\delta]$ is reduced, by \cite[Theorem
3.3]{krempa}.
\end{ex}

\section*{acknowledgments}The authors express their gratitude to Professor
Gary F. Birkenmeier for valuable suggestions and helpful comments.
Also they are deeply indebted to the referee for many helpful
comments and suggestions for the improvement of this paper.


\begin{thebibliography}{15}
\bibitem{birk1} G.F. Birkenmeier, J.Y. Kim, J.K. Park, {\it Polynomial extensions of Baer
and quasi-Baer rings}, J. Pure appl. Algebra 159: (2001) 25-42.
\bibitem{birk2}G.F. Birkenmeier, J.Y. Kim, J.K. Park, {\it Principally quasi-Baer rings}, Comm. Algebra
29(2): (2001) 639-660.
\bibitem{birk3} G.F. Birkenmeier, B.J. M\"uller, S.T. Rizvi, {\it Modules in which every fully invariant submodule is
issential in a direct summand}, Comm. Algebra 30(3): (2002)
1395-1415.
\bibitem{Hashemi/polext} E. Hashemi, A. Moussavi, {\it Polynomial extensions of quasi-Baer rings}
, Acta math. Hungar. 107(3): (2005) 207-224.
\bibitem{Hashemi/skewarm} E. Hashemi, A. Moussavi, {\it On $(\sigma,\delta)$-skew Armendariz rings}
, J. Korean Math. soc. 42(2): (2005) 353-363.
\bibitem{birk4} J. Han,  Y. Hirano,  H. Kim, {\it Some results on skew polynomial rings over a reduced ring},
In: G.F. Birkenmeier, J.K. Park, Y.S. Park (Eds), The international
symposium on ring theory, In: Trends in math., Birkh\"auser Boston
(2001).
\bibitem{hong1} C.Y. Hong, N.K. Kim, T.K. Kwak, {\it Ore extensions of Baer and p.p.-rings},
J. Pure and Appl. Algebra 151(3): (2000) 215-226.
\bibitem{hong2} C.Y. Hong, N.K. Kim, T.K. Kwak, {\it On Skew Armendariz Rings},
Comm. Algebra 31(1): (2003) 103-122.
\bibitem{Kaplansky} I. Kaplansky, {\it Rings of operators}, Math.lecture Notes series
, Benjamin, New York (1965).
\bibitem{krempa}J. Krempa, {\it Some examples of reduced rings},
Algebra Colloq. 3(4): (1996) 289-300.
\bibitem{Lam} T.Y. Lam, A. Leroy, J. Matczuk, {\it Primeness, semiprimeness and the
prime radical of Ore extensions}, Comm. Algebra 25(8): (1997)
2459-2506.

\end{thebibliography}
\end{document}